\documentclass[12pt]{amsart}
\usepackage{amsfonts}

\usepackage[letterpaper]{geometry}
\usepackage{amsmath}
\usepackage{mathtools}
\usepackage{commath}

\usepackage{amssymb}
\def\Z{\mathbb{Z}}
\def\N{\mathbb{N}}
\def\S{\mathbb{S}}

\newtheorem{thm}{\bf Theorem}[section]
\newtheorem{lemma}[thm]{\bf Lemma}

\theoremstyle{definition}

\begin{document}

\title[On the distribution of subset sums of certain sets  in  $\Z^2_p$]{On the distribution of subset sums of certain sets  in $\Z^2_p$}

\author[]{Norbert Hegyv\'ari}
\address{Norbert Hegyv\'{a}ri, ELTE TTK,
E\"otv\"os University, Institute of Mathematics, H-1117
P\'{a}zm\'{a}ny st. 1/c, Budapest, Hungary and Alfr\'ed R\'enyi Institute of Mathematics, H-1364 Budapest, P.O.Box 127.}
 \email{hegyvari@renyi.hu}

\maketitle

\begin{abstract}
A given subset $A$ of natural numbers is said to be complete if every element of $\N$ is the sum of distinct terms taken from $A$. This topic is strongly connected to the knapsack problem which is known to be NP complete. 

Interestingly if $A$ and $B$ are complete sequences then $A\times B$ is not necessarily complete in $\N^2$. In this paper we consider a modular version of this problem, motivated by the communication complexity problem of [2].

AMS 2010 Primary 11B30, 11B39, Secondary 11B75

Keywords: Subset sums, additive combinatorics, matching in special graph 
\end{abstract}

\section{Introduction}
Let $\S$ be any additive semigroup, For sets $A,B\subset \S$ the sum is defined by $A \pm B:=\{a\pm b:a\in A; \ b\in B\}$, 
For any $X\subseteq \S$ let
$$
FS(X):=\{\sum_{i=1}^\infty\varepsilon_ix_i: \ x_i\in X, \ \varepsilon_i \in \{0,1\}, \ \sum_{i=1}^\infty\varepsilon_i<\infty\},
$$
and one can extend the notion of subset sums $FS(X)$ for higher dimension, i.e. when $X\subseteq \S^k; k\geq 1$.
A set $X\subseteq \S$ is said to be complete respect to $\S$ if every element of $\S$ is the sum of distinct terms taken from $X$. If $\S=\N$, there is a wide literature on the subject (A good survey can be found e.g. in [6], [10] and [12]).

This type of problems are highly connected to the knapsack or subset-sum problem. The problem is to determine, given positive integers $a_1,a_2,\dots,a_n$ and an integer $m$, whether there is a subset of the set $\{a_j\}$ that sums up to $m$. It is well known as an NP-complete problem. This topic is connected to some communication  complexity problem  too (see e.g. [2]) which motivates the present work. 

Completeness in higher dimensions is very different from the one-dimensional case. For example, a classic result is that if a series contains more than $c\sqrt{n}$ elements up to $n>n_0$, then it is complete (see [12]). While in the two-dimensional case, one can give a dense set $A$ for which $FS(A)$ does not even contain an arithmetic progression (see [8] and [9]). Further results can be found in [1] and [4].

\smallskip

Let $F_n$ be the $n^{th}$ Fibonacci number, i.e. let $F_0=0; \ F_1=1$ and for $n\geq 0$ $F_{n+2}=F_{n+1}+F_n$. For negative integers let $F_{-n}=(-1)^{n-1}F_n$ when the recursive formula also holds. In section 2 we will investigate the distribution of the subset sums of three sets in $\Z^2_p$.  

More precisely let $A_1=\{2^n\}_{n=0}^\infty \times \{F_n\}_{n=0}^\infty$ , $A_2=\{2^n\}_{n=0}^\infty\times \{2^n\}_{n=0}^\infty$ and $A_3=  \{F_n\}_{n=0}^\infty\times  \{F_n\}_{n=0}^\infty$. We will show that for almost all prim $p$, $FS(A_i) \pmod p=\Z^2_p$ ($i=1,2,3$). For $X\subseteq \N^2$ $X\pmod p=\{(v_1,v_2)\in \Z^2_p: \exists (x_1,x_2)\in X; \ v_i\equiv x_i\pmod p; \ i=1,2 \}$. 

Furthermore we extend the analysis to sets of types $A_4=\{b^n\}_{n=0}^\infty \times \{F_n\}_{n=0}^\infty$ and $A_5=\{b^n\}_{n=0}^\infty \times \{2^n\}_{n=0}^\infty$, where $b>2$ integer. Although the set $\{b^n\}_{n=0}^\infty$ is incomplete for $b>2$, note that there can be an unbounded number of times $b^n$ as one of the coordinates in the sets $A_4$ and $A_5$. Thus, since all positive integers can be represented in the $b$-ary expansion, this question fits the previous ones.

In [2] we also consider the above questions in $\N^2$. Interestingly, the distribution of the above sets is different than in $\Z^2_p$. For example, for $i=1,2$ and $3$, we describe sets $E_i$ for which $\N^2\setminus E_i \subseteq FS(A_i)$. Furthermore, we show that $E_i$ contains arbitrary "large" squares which contain no elements from $FS(A_i)$, and  squares of size $k \times k$ which contain $k\log k$ many elements from $FS(A_i)$.  

\section{Modular case}

In this section we will prove the following theorem:
\begin{thm}\label{b}
There exists a $\delta>0$ such that for all but $O(x/(\log x)^{1+\delta})$ primes $p\leq x$, $|FS(A_4)\pmod p| \geq \frac{p^2}{b}.$
\end{thm}
When $b=2$, the previous theorem says that $|FS(A_1)\pmod p| \geq \frac{p^2}{2}.$ But in this case we get more:
\begin{thm}\label{M}
There exists a $\delta>0$ such that for all but $O(x/(\log x)^{1+\delta})$ primes $p\leq x$, $FS(A_1) \pmod p=\Z^2_p$. 
\end{thm}
The following theorem can be proved in a similar way as Theorem \ref{M}
\begin{thm}
There exists a $\delta>0$ such that for all but $O(x/(\log x)^{1+\delta})$ primes $p\leq x$, $FS(A_i) \pmod p=\Z^2_p$ ($i=2,3$). 
\end{thm}

\begin{proof}[Proof of Theorem \ref{M}] 
Let $\mathcal{P}_1$ be the sets of all primes up to $x$ for which $|\pmod p\{2^n\}_{n=0}^\infty|>2\sqrt{p}$ and let $\mathcal{P}_2$ be the sets of all primes up to $x$ for which $|\pmod p\{F_n\}_{n=0}^\infty|>2\sqrt{p}$. Denote by $\mathcal{P}=\mathcal{P}_1\cap \mathcal{P}_2$. We will check that for all but $O(x/(\log x)^{1+\delta})$ primes $p\leq x$ belong to $\mathcal{P}$.

We need the following lemmas
\begin{lemma}\label{2.3}
There exists a $\delta>0$ such that for all but $O(\frac{x}{(\log x)^{1+\delta}})$ primes $p\leq x$ we have
\begin{equation}\label{F}
|\{F_n\}_{n=0}^\infty \pmod p|>3\sqrt{p},
\end{equation}
where $\delta>0.086072$ is admissible if $x$ is large enough.
\end{lemma}
It is a slight modification of [7, Lemma 3.1].
\begin{lemma}\label{2.2}
 There exist a $\delta>0$ such that 
\begin{equation}\label{2}
|\{b^n\}_{n=0}^\infty \pmod p|>2\sqrt{p}
\end{equation}
for all but $O(x/(\log x)^{1+\delta})$ primes $p\leq x$. 
One can choose $\delta$ as in the previous lemma.
\end{lemma}
This lemma is a simple consequence of Theorem 3 in [5].

Using these lemmas we get that for all but $O(x/(\log x)^{1+\delta})$ primes $p\leq x$ both (\ref{2}) and (\ref{F}) hold. Denote the set of these primes by $\mathcal{P}$ and we get that for every $p\in \mathcal{P}$ both $T:=|\pmod p\{b^n\}_{n=0}^\infty|>2\sqrt{p}$ and $R:=|\pmod p\{F_n\}_{n=0}^\infty|>2\sqrt{p}$ hold. (We only need the factor $3$ of the Lemma \ref{2.3} when we assume that $b>2$, see section 2.1). The third ingredient of our proof is the following 
\begin{lemma}[Olson]\label{O}
Let $A=\{a_1,\dots, a_k\}$ be distinct non-zero residue classes modulo the prime $p$. If $k>2\sqrt{p}$ then $FS(A)=\Z_p$.
\end{lemma}
See the proof in [11].
From now on in this section we assume that $b=2$. 
Write $\pmod p\{F_n\}_{n=0}^\infty=\{F_{j_1},F_{j_2},\dots, F_{j_R}\}$ and $\pmod p\{2^n\}_{n=0}^\infty=\{2^{i_1},2^{i_2},\dots, 2^{i_T}\}$. 
 Let $(v_1,v_2)$ any point in $\Z^2_p$. By Lemma \ref{O}  we can write $v_1$ in the form $v_1=\sum^T_{m=1}\varepsilon_m2^{i_m}; \ \varepsilon_m\in \{0,1\}$, and $v_2$ in the form $v_2=\sum^R_{k=1}\eta_kF_{j_k}; \ \eta_k\in \{0,1\}$.

Assume first that $\sum^T_{m=1}\varepsilon_m>\sum^R_{k=1}\eta_k$. We are going to find a subset $\{F'_{j_n}\}$ for which $\sum^V_{n=1}\zeta_nF'_{j_n}=v_2$, $\zeta_n>0$ and $\sum_n\zeta_n=T$. Note that we allow $\zeta_n$ to be bigger than one, i.e. we allow some $F'_j$ to be repeated. Then  the elements $(2^{i_1},F'_{j_1} ),(2^{i_2},F'_{j_2}),\dots, (2^{i_T},F'_{j_T})$ are pairwise different (because their first coordinates are pairwise different) and $\sum^T_{t=1}(2^{i_t},F'_{j_t})=(v_1,v_2)$, as we wanted.

So take $F_{i_1}$ and find an element $F_r$, where $r>\sum^T_{m=1}\varepsilon_m$ and $F_r\equiv F_{i_1}\pmod p$. Since the sequence ${F_n}$ is periodic $\pmod p$, there exists such an $r$. Now change $F_r$ to $F_{r-1}+F_{r-2}$. Repeat this changing process altogether 
 $\sum^T_{m=1}\varepsilon_m-\sum^R_{k=1}\eta_k$ many times (i.e. change $F_{r-1}$ to $F_{r-2}+F_{r-3}$, etc.). Since $r>\sum^T_{m=1}\varepsilon_m$ this process is terminated. Then we can matching the terms of powers of two to the (possible repeated) elements of the appropriate Fibonacci elements.

 When $\sum^T_{m=1}\varepsilon_m<\sum^R_{k=1}\eta_k$ the process is the same; just note that we replace a term $2^m$ by $2^{m-1}+2^{m-1}$.
\end{proof}

\subsection{The case when $b>2$} In the general case we give a bound for the cardinality of $FS(A_4) \pmod p$, where $A_4=\{b^n\}_{n=0}^\infty \times \{F_n\}_{n=0}^\infty$ 
\begin{proof}[Proof of Theorem \ref{b}]
Let $p\in  \mathcal{P}$. Split $\pmod p\{F_n\}_{n=0}^\infty=\{F_{j_1},F_{j_2},\dots, \}$ into two disjoint parts $G_1\cup G_2$ where $|G_2|=b-1$. Write $\pmod p\{b^n\}_{n=0}^\infty=\{b^{i_1},b^{i_2},\dots, b^{i_T}\}$, $G_1=\{F{i_1},F_{i_2},\dots, F_{i_V}\}$ and $G_2=\{F_{s_1},F_{s_2},\dots, F_{s_{b-1}}\}$. By the assumption and Lemma \ref{2.2} we get that $V=|G_1|>2\sqrt{p}$. By Lemma \ref{O} for any $(v_1,v_2)\in \Z^2_p$ $v_1$ can be written as $\sum^T_{k=1}\varepsilon_kb^{i_k};\ \varepsilon_k\in \{0,1\}$, and $v_2$ in the form $v_2=\sum^V_{k=1}\eta_kF_{j_k}; \ \eta_k\in \{0,1\}$.

Let us assume that
$\sum^T_{m=1}\varepsilon_m>\sum^V_{k=1}\eta_k$. Then follow the previous process. 


So let us assume that $\sum^T_{m=1}\varepsilon_m<\sum^V_{k=1}\eta_k$. Then take an $m$ large enough for which $b^{m}\equiv b^{i_1}\pmod p$, and change $b^m=b^{m-1}+b^{m-1}+\dots +b^{m-1}$($m$ times). Repeat this process increasing the number of terms in each step by $b$ elements. If one step the number of terms is exactly $\sum^V_{k=1}\eta_k$, then we are done. Otherwise one repetition we overstep $\sum^V_{k=1}\eta_k$ by $r$ many elements and clearly $\min r\leq b-1$. Then completing the sum $\sum^V_{k=1}\eta_kF_{j_k}$ by the terms $F_{s_1},F_{s_2},\dots, F_{s_r}$ we can represent $(v_1,v_2+\sum^r_{t=1}F_{s_t})$ in $FS(A_4)\pmod p$.

Consider a coloring of $\Z^2_p$ according the values of $\sum^r_{t=1}F_{s_t}; (r=0,1,\dots b-1)$, i.e. let $(v_1,v_2)$ be color $r$, if $(v_1,v_2+\sum^r_{t=1}F_{s_t})$ in $FS(A_4)\pmod p$. Note that the coloring is not unique; $(v_1,v_2)$ and $(v'_1,v'_2)$ can have the same color. For clarity, choose the color $r$ to be minimal. Denote by $L_r$, the set of points of color $r$. Now choose the $L_{r_0}\subseteq \Z^2_p$ for which $|L_{r_0}|$ is maximal. We have $|L_{r_0}|\geq \frac{p^2}{b}$. Then $L_{r_0}+\sum^{r_0}_{t=1}F_{s_t}\subseteq FS(A_4)\pmod p$, and $|L_{r_0}+\sum^{r_0}_{t=1}F_{s_t}|=|L_{r_0}||\geq \frac{p^2}{b}$.

\end{proof}
\section*{Acknowledgment}

The research is supported by the National Research, Development and Innovation Office NKFIH Grant No K-129335.


\begin{thebibliography}{99}

\bibitem[1]{1} B. Bakos, P\'alfy, Some results on an encryption method using subset-sums of pseudo-recursive sequences
Discrete Mathematics Letters 5 : 1 pp. 63-67. , 5 p. (2021)

\bibitem[2]{2} B. Bakos, N. Hegyv\'ari and M. P\'alfy, On a communication complexity problem in combinatorial number theory, Moscow Journal of Combinatorics, Vol. 10 (2021), No. 4, 297–302 DOI: 10.2140/moscow.2021.10.297

\bibitem[3]{3} B. Bakos, N. Hegyv\'ari and M. P\'alfy On the distribution of subset sums of certain sets  in  $\N^2$, in progress

\bibitem[4]{4} Chen, YG ; Fang, JH ; Hegyv\'ari, N,
Erd\H os-Birch type question in $\N^r$
J. of Number Theory 187 pp. 233-249. , 17 p. (2018)

\bibitem[5]{5} P. Erd\H os and M. R. Murty, On the order of a mod p, in: Number Theory (Ottawa, ON, 1996),
CRM Proceedings and Lecture Notes, 19 (American Mathematical Society, Providence, RI, 1999),
pp. 87–97.

\bibitem[6]{6} P. Erd\H os, R.L. Graham, Old and new problems and results in combinatorial number theory, Monographies de L'Enseignement Mathématique

\bibitem[7]{7} Garcia V, Luca F and Mejia J., On sums of Fibonacci numbers modulo p, Bulletin of the Australian Mathematical Society 83(03):413-419, DOI:10.1017/S0004972710001693

\bibitem[8]{8} Hegyv\'ari, N. Complete sequences in $\N^2$ European Journal of Combinatorics (17) (1996)  741-749. 

\bibitem[9]{9} Hegyv\'ari, N Subset sums in $\N^2$ Combinatorics Probability and Computing (5) (1996) 393-402.


\bibitem[10]{10} N.B. Nathanson, Additive Number Theory Springer-Verlag New York 1996

\bibitem[11]{11} J. E. Olson, An addition theorem modulo p, J. Combinatorial Theory 5(1968), 45-52

\bibitem[12]{12} T. Tao, V. Vu,  \emph{Additive Combinatorics.} Cambridge University Press, Cambridge 2006



\end{thebibliography}
\end{document}